\documentclass[intlimits]{article}
\usepackage{latexsym,amsfonts,amssymb,amsmath,amsthm,amscd}
\usepackage{graphicx}
\usepackage{float}

\theoremstyle{plain}
\newtheorem{theorem}{Theorem}[section]
\newtheorem{lemma}{Lemma}[section]
\newtheorem{proposition}{Proposition}[section]

\newtheorem{remark}{Remark}[section]

\numberwithin{equation}{section}

\newcommand\RR{\ensuremath{\mathbb{R}}}

\newcommand\ZZ{\ensuremath{\mathbb{Z}}}
\newcommand\NN{\ensuremath{\mathbb{N}}}

\newcommand\CC{\ensuremath{\mathbb{C}}}

\newcommand{\norm}[1]{\ensuremath{\lVert#1\rVert}}

\providecommand{\keywords}[1]{\medskip \noindent \textbf{Keywords.} #1}
\providecommand{\amsclass}[1]{\medskip \noindent \textbf{AMS subject classification.} #1}

\DeclareMathOperator{\coshsq}{cosh^{2}}
\DeclareMathOperator{\sinhsq}{sinh^{2}}

\begin{document}

\title{Instability in linear cooperative systems of ordinary differential
equations}

\author{Janusz Mierczy\'nski
\\
Faculty of Pure and Applied Mathematics
\\
Wroc{\l}aw University of Science and Technology
\\
Wybrze\.ze Wyspia\'nskiego 27
\\
PL-50-370 Wroc{\l}aw
\\
Poland
}
\date{}

\maketitle

\begin{abstract}
It is well known that, contrary to the autonomous case, the stability/instability of solutions of nonautonomous linear ordinary differential equations $x' = A(t) x$ is in no relation to the sign of the real parts of the eigenvalues of $A(t)$.  In~particular, the real parts of all eigenvalues can be negative and bounded away from zero, nonetheless there is a solution of magnitude growing to infinity.

In this paper we present a method of constructing examples of such systems when the matrices $A(t)$ have positive off-diagonal entries (strongly cooperative systems).  We illustrate those examples both with interactive animations and analytically.  The paper is written in such a way that it can be accessible to students with diverse mathematical backgrounds/skills.
\end{abstract}


\amsclass{34C12, 34D99, 34A26, 34A40}

\section{Introduction}
\label{sec:introduction}

It is a well-known fact that for an autonomous system of linear
ordinary differential equations (ODEs)
\begin{equation*}
x' = Ax,
\end{equation*}
where $A$ is a {\em constant\/} $n$ by $n$ matrix with real entries, the zero solution is asymptotically stable if and only if the real parts of the eigenvalues of $A$ are negative.

Unfortunately, for nonautonomous systems of linear ODEs
\begin{equation}
\label{general-nonautonomous}
x' = A(t)x
\end{equation}
there is no hope for a similar result.  Indeed, one can find examples of systems~\ref{general-nonautonomous} such that for all $t$ all the eigenvalues of $A(t)$ are negative but there is a solution of~\ref{general-nonautonomous} whose norm tends to infinity as $t \to \infty$.   Some of those examples, although not a part of the standard curriculum, have made their way to textbooks, see, e.g. Example~III.7.1 in~\cite{Hale}.  For a nice paper on that subject, see~\cite{JoRo}.

\bigskip
The purpose of the present paper is to give a method for finding such examples when the linear system is {\em strongly cooperative\/}:  for each $t \in \RR$ the matrix $A(t)$ has positive off-diagonal entries.

The linear (and not only linear) strongly cooperative systems are of interest in itself, see, e.g., \cite{Sm} or \cite{HiSm}.  However, let us concentrate now on their biological relevance. For~instance, in some bacterial populations there is switching between two states (dormant vs.\ active).  It should be  remarked here that, to be sure, ``real life'' models are usually nonlinear, but a linear two-dimensional system, $x' = A(t) x$, can serve as a first approximation. If we let $x_1$ stand for the density of bacteria in the dormant stage and $x_2$ for the density of bacteria in the active stage, $a_{12}(t)$ (resp.\ $a_{21}(t)$) describes the transition rate from the active to the dormant state (resp.\ from the dormant into the active state) at time $t$. It is straightforward that $a_{12}(t)$ and $a_{21}(t)$ must be nonnegative for each $t$. See, for~example, \cite{MaSm}.

\smallskip
It is the survival of the population that is frequently of interest to us.  A mathematical expression of that survival is the notion of permanence.  Not delving into the details, this means that, however small the initial population is, after sufficiently long time it becomes and stays bounded away from zero, with the bound independent of the initial value.  In linear models this means simply that the magnitude of a solution tends to infinity as time goes to infinity.  One should bear in mind that the divergence to infinity is, in itself, a spurious artifact, as the (linear) model loses any relevance for large population densities.

\bigskip

In the main part of the paper, section \ref{sec:construction}, we give a construction of a linear time-periodic strongly cooperative two-dimensional system $x' = A(t) x$ of ordinary differential equations such that the larger of the (necessarily real) eigenvalues of $A(t)$ (the {\em principal eigenvalue\/}) is equal to $-1/2$ at all $t$ yet there exists an unstable solution.

The idea of our construction is the following:  during the first half of period, the evolution of the system is governed by a (far from symmetric) constant matrix having its eigenvector with both coordinates positive (the {\em principal eigenvector\/}) close to one coordinate axis, whereas during the second half of period, the evolution of the system is governed by another (again far from symmetric) constant matrix having its principal eigenvector close to the other coordinate axis; around the half-period there occurs a very fast (in section \ref{sec:construction} instantaneous) change in matrices.  It should be emphasized here that periodicity is not a necessary feature of the construction: Its r\^ole is rather to streamline the argument.  Analogous examples can be obtained for other linear systems that are in an appropriate sense recurrent in time, whether nonautonomous or random (an example is given in section \ref{sec:non-periodic}).

\medskip
The paper is written in a reasonably self-contained way.  It is assumed that the reader knows the standard facts from linear ordinary differential equations (transition matrices, etc.).  However, the knowledge of time-dependent (or even time-periodic) linear ordinary differential equations is not indispensable (except section \ref{sec:continuous}): As the systems considered are piecewise constant,  knowing basic properties of the matrix exponentials should suffice.

In subsection \ref{subsec:ODE-systems} we give a review of standard results on the properties of solutions of systems of linear ordinary differential equations, not necessarily autonomous, while in section \ref{subsec:exp-matrix} we present results on the exponents of matrices.

Subsection \ref{subsec:action-on-unit-circle} is devoted to analytical study of the action of the matrix exponential on the lengths and directions of vectors.

As the material related to matrices with positive off-diagonal entries does not usually form a part of the curriculum, in subsection \ref{subsec:positive-matrices} we give necessary proofs.  While this usually requires using rather advanced methods (see, e.g., \cite{BP}), in our case needed proofs are given by using only the knowledge in calculus and elementary algebra.

Subsection \ref{subsec:strongly-cooperative} deals with a proof of (the linear specialization of) a celebrated result due to M\"uller and Kamke on the order preserving property of quasimonotone systems (see~\cite{HiSm}).  Indeed, two alternative proofs of that property are given.  The first (the proof of \ref{thm:strongly-monotone}) uses tools from calculus and is decidedly nonlinear in its spirit.  However, its full strength is used only in ection \ref{sec:continuous}.  An alternative is the proof of \ref{thm:strongly-monotone-nonautonomous}, which uses only basic properties of the exponential of a matrix.

In Subsection \ref{subsec:action-on-unit-circle_continued} we continue our analysis of the action of the matrix exponential on the vector directions from subsection \ref{subsec:action-on-unit-circle}, this time under the assumption that the matrix has positive off-diagonal entries.  The material is illustrated by pictures and animations.

\smallskip
After all those preliminaries we proceed in section \ref{sec:construction} to give the construction of our example. We give first an idea and then explain, assisted by pictures, why the construction should be O.K.   Then, in subsubsection \ref{subsubsec:analysis} we give two alternative ``hard'' analytical proofs of the existence of an unstable solution.   The first proof rests on direct computation of the largest eigenvalue of a transition matrix, and requires only the knowledge of the fundamental properties of matrix exponential.   The second proof uses the Peano--Baker series.

\smallskip
Section \ref{sec:continuous} requires more advanced knowledge (however, not reaching beyond the Gronwall inequality or matrix norms).  It can serve as a basis for some undergraduate homework.

\medskip
In section \ref{sec:extensions} we give a couple of extensions and generalizations (which can again be the subject of some undergraduate work).  Section \ref{sec:non-periodic} provides extension to the case of non\nobreakdash-\hspace{0pt}periodic systems.  Its reading requires the knowledge of standard calculus.

\medskip
Finally, section \ref{sec:overview} (Discussion) puts the material presented in the perspective of what is already known.

\section{Preliminaries}
\label{sec:preliminaries}

\subsection{Systems of linear ODEs}
\label{subsec:ODE-systems}

Consider a system of two linear ODEs
\begin{equation}
\label{eq:ODE}
x' = A(t)x,
\end{equation}
where we assume that $A \colon J \to \RR^{2 \times 2}$ is a continuous matrix function ($J \subset \RR$ is an interval not reducing to a singleton, and $\RR^{2 \times 2}$ denotes the set of real $2$ by $2$ matrices).

It is a standard result in the course in ODEs that for each $s \in J$ and each $x_0 \in \RR^2$ there exists a unique solution, $x(\cdot; s, x_0)$, of the initial value problem
\begin{equation*}
\begin{cases}
x' = A(t)x \\
x(s) = x_0,
\end{cases}
\end{equation*}
and that solution is defined on the whole of $J$.

Usually stress is laid on fundamental matrices (cf.\ \cite{C} or~\cite{Hale}):  $X(\cdot)$ is a fundamental matrix solution of~\ref{eq:ODE} if its columns form a basis of the vector space of solutions of~\ref{eq:ODE}.  For our purposes, however, it is better to use the {\em transition matrix\/} (see~\cite{C}), that is, a matrix function of two variables, $X = X(t;s)$, $s, t \in J$, such that for any $s \in J$ and any $x_0 \in \RR^2$, there holds
\begin{equation*}
x(t; s, x_0) = X(t; s) x_0, \qquad t \in J.
\end{equation*}
If $X(\cdot)$ is a fundamental matrix solution, the transition matrix is given by the formula
\begin{equation}
\label{eq:fundamental-to-transition}
X(t; s) = X(t) X^{-1}(s), \qquad s, t \in J.
\end{equation}
The transition matrix is unique.

We mention here important properties of the transition matrix:
\begin{proposition}
\label{prop:transition-properties}
\begin{enumerate}
\item[\textup{(1)}]
$X(s; s) = I$, for any $s \in J$, where $I$ is the identity
matrix;
\item[\textup{(2)}]
$X(u;s) = X(u; t) X(t; s)$, for any $s, t, u \in J$;
\item[\textup{(3)}]
$X^{-1}(t;s) = X(s; t)$, for any $s, t \in J$.
\item[\textup{(4)}]
$\displaystyle \frac{\partial}{\partial t} X(t;s) = A(t) X(t;s)$,
for any $s, t \in J$.
\end{enumerate}
\end{proposition}

\subsection{Systems of autonomous linear ODEs.  The matrix $e^{t A}$}
\label{subsec:exp-matrix}

In modern courses in ODEs, when considering systems of autonomous linear ordinary differential equations
\begin{equation}
\label{eq:ODE-autonomous}
x' = A x,
\end{equation}
usually a matrix function $t \mapsto e^{tA}$ is introduced, where
\begin{equation*}
e^{tA} := \sum_{k=0}^{\infty} \frac{t^{k} A^{k}}{k!}.
\end{equation*}
Occasionally, for typographical reasons we write $\exp(tA)$ instead~of $e^{tA}$.  It is proved that the above series has convergence radius infinity, the function $t \mapsto e^{tA}$ is differentiable, and the relations
\begin{itemize}
\item
$e^{0 \cdot A} = I$,
\item
$e^{(s + t)A} = e^{sA} e^{tA}$, $s, t \in \RR$,
\item
$(e^{tA})^{-1} = e^{-tA}$, $t \in \RR$,
\item
$\dfrac{d}{dt} e^{tA} = A e^{tA} = e^{tA} A$, $t \in \RR$
\end{itemize}
hold. Consequently, the solution of the initial-value problem for a system of ordinary
differential equations with time-independent matrix $A$,
\begin{equation}
\label{eq:autonomous-IVP}
\begin{cases}
x' = Ax,
\\
x(0) = x_0
\end{cases}
\end{equation}
equals
\begin{equation*}
e^{tA} x_0.
\end{equation*}

We would like to put the above into the context of transition matrices.  Since the matrix function $e^{tA}$ is (a special case of) a fundamental matrix solution of~\ref{eq:ODE-autonomous},
by using the formula~\ref{eq:fundamental-to-transition} we obtain
\begin{equation}
\label{eq:transition-autonomous}
X(t; s) = X(t) X^{-1}(s) = e^{tA}  (e^{sA})^{-1} = e^{tA}  e^{-sA} = e^{(t - s) A}, \quad s, t \in \RR.
\end{equation}

\smallskip
We will use in the sequel the following fact.
\begin{equation}
\label{eq:commute}
\text{If } AB = BA \text{ then } e^{t (A + B)} = e^{tA} e^{tB} =
e^{tB} e^{tA} \text{ for all } t \in \RR.
\end{equation}
However, for general $A, B$ the above equalities need not hold.

\subsection{The action of $e^{tA}$ on the unit circle}
\label{subsec:action-on-unit-circle}
In this subsection we shall analyze how the radiuses and directions of solutions of the system $x' = Ax$ change in~time.  In other words, we investigate the action of $e^{tA}$ on vectors in $\RR^2$.

\smallskip
We start by introducing some notation.

Recall that we can represent $x \in \RR^{2}$ in polar coordinates, $x = r \, [\cos{(\theta)} \ \ \sin{(\theta)}]^{\top}$, where $r = \norm{x} = \sqrt{x \cdot x}$ is the {\em length\/} (magnitude, norm) and $\theta$ is the {\em polar angle\/} of $x$.

We denote by $\mathbb{S}$ the set of all vectors $y \in \RR^2$ with
unit length.  In other words, $\mathbb{S}$ is the unit circle.

\medskip
Let $x(t)$ be a nontrivial (that is, not equal constantly to zero) solution of $x' = A x$.  That is, $x(t) = e^{tA}x_0$ for some nonzero $x_0$.

\subsubsection{How does $e^{tA}$ act on the lengths of vectors?}
\label{subsubsect:acting-on-lengths}

As a warm-up we try to find an ordinary differential equation satisfied by $\norm{x(t)}$. After some calculus we obtain
\begin{equation}
\label{eq:action-exp-radius}
\begin{aligned}
\frac{d}{dt} \norm{x(t)} = & {} \frac{d}{dt} (x(t) \cdot x(t))^{1/2}
\\
= & {} \frac{1}{2} \frac{x'(t) \cdot x(t) + x(t) \cdot x'(t)}{(x(t) \cdot x(t))^{1/2}} = \frac{Ax(t) \cdot x(t)}{\norm{x(t)}}.
\end{aligned}
\end{equation}

\subsubsection{How does $e^{tA}$ act on the directions of vectors?}
\label{subsubsect:acting-on-directions}

The present subsubsection can be skipped, since it will be needed later only for heuristic considerations in subsubsection \ref{subsubsec:heuristics}.

Let us find an ordinary differential equation that is satisfied by the direction of $x(t)$.  We differentiate
\begin{equation*}
\begin{aligned}
\frac{d}{dt} \frac{x(t)}{\norm{x(t)}} = {} & \frac{d}{dt} \Bigl(
\bigl( x(t) \cdot x(t) \bigr)^{-1/2} x(t) \Bigr)
\\
= {} & \Bigl( \frac{d}{dt} \bigl( x(t) \cdot x(t) \bigr)^{-1/2} \Bigr) \, x(t) + \bigl( x(t) \cdot x(t) \bigr)^{-1/2} \frac{d}{dt} x(t)
\\
= {} & -\frac{1}{2} \bigl( x(t) \cdot x(t) \bigr)^{-3/2} \, 2 \bigl( A x(t) \cdot x(t) \bigr) x(t) + \bigl( x(t) \cdot x(t) \bigr)^{-1/2} A x(t)
\\
= {} & \frac{1}{\norm{x(t)}} \biggl( A - \frac{A x(t) \cdot x(t) }{\norm{x(t)}^2} I \biggr) x(t) = \biggl( A - \Bigl( A \frac{x(t)}{\norm{x(t)}} \cdot \frac{x(t)}{\norm{x(t)}} \Bigr) I \biggr) \frac{x(t)}{\norm{x(t)}},
\end{aligned}
\end{equation*}
or, after putting $y(t) := x(t)/\norm{x(t)}$,
\begin{equation*}
\frac{d}{dt} y(t)  = \bigl( A - (A y(t) \cdot y(t)) I \bigr) y(t).
\end{equation*}
We can say that $y(t)$ is a solution of a system of two (nonlinear) ordinary differential equations, written in the matrix form as
\begin{equation}
\label{eq:polar-angle}
y'  = \bigl( A - (A y \cdot y) I \bigr) y.
\end{equation}
Observe that for any $y \in \mathbb{S}$ the vector $\bigl( A - (A y \cdot y) I \bigr) y$ is perpendicular to $y$. Indeed, there holds
\begin{equation*}
\bigl( A - (A y \cdot y) I \bigr) y \cdot y = A y \cdot y -  (A y \cdot y) y \cdot y = (A y \cdot y) ( 1 - \norm{y}^2) = 0.
\end{equation*}
It follows that for a solution $y(t)$ of~\ref{eq:polar-angle} we have
\begin{equation*}
\frac{d}{dt} \norm{y(t)}^2 = 2 (y'(t) \cdot y(t)) =  \bigl( A - (A y(t) \cdot y(t)) I \bigr) y(t) \cdot y(t)  = 0,
\end{equation*}
from which we can conclude that, if at an initial moment $s$ the value $\norm{y(s)}$ is equal to one then it is equal to one at any time.  So, although system~\ref{eq:polar-angle} is well defined for all $y \in \RR^2$, we will consider it for $y$ belonging to $\mathbb{S}$ only.

\medskip
Let us find what are the equilibria of~\ref{eq:polar-angle}, that is, those $\eta \in \mathbb{S}$ for which $\bigl( A - (A \eta \cdot \eta) I \bigr) \eta = [0 \ \ 0]^{\top}$.  We have then $A \eta = (A \eta \cdot \eta) \eta$, which translates into $\eta$ being an eigenvector of the matrix $A$, corresponding to an eigenvalue $A \eta \cdot \eta$.

If $y \in \mathbb{S}$ is not an eigenvector of $A$ then the nonzero vector $\bigl( A - (A y \cdot y) I \bigr) y$, perpendicular to $y$, points either clockwise or counterclockwise.

\medskip
We will return later, in Subsection~\ref{subsec:action-on-unit-circle_continued}, to analyzing the action of $e^{tA}$.

\subsection{Matrices with positive off-diagonal entries\ ---\ Their spectral
properties}
\label{subsec:positive-matrices}

Since, as mentioned in the Introduction, linear differential equations having  matrices with positive off-diagonal entries are our main object of study, we give in the present subsection some useful information on spectral properties of such 2 by 2 matrices.

We write $\RR^{2}_{+}$ for the set of all those $x = [x_1 \ \ x_2]^{\top}$ such that $x_1 \ge 0$ and $x_2 \ge 0$, and $\RR^{2}_{++}$ for the set of all those $x = [x_1 \ \ x_2]^{\top}$ such that $x_1 > 0$ and $x_2 > 0$.

Let $\mathcal{M}$ stand for the family of real $2 \times 2$ matrices having their off-diagonal entries positive, and let $\mathcal{P}$ stand for the family of real $2 \times 2$ matrices having all entries positive.

Our first result is usually known as the Frobenius--Perron theorem. As we are in dimension two, we will give here an elementary proof of it.
\begin{proposition}
\label{prop:principal}
Let $A = [a_{ij}]_{i,j=1}^2 \in \mathcal{M}$.  Then the following holds:
\begin{enumerate}
\item[{\rm (i)}]
$A$ has two real eigenvalues \textup{(}denoted $\lambda_2 < \lambda_1$\textup{)}.
\item[{\rm (ii)}]
An eigenvector $u$ corresponding to $\lambda_1$ can be taken to have its coordinates positive.
\item[{\rm (iii)}]
$\lambda_1 > \max\{a_{11}, a_{22}\}$, and $\lambda_2 < \min\{a_{11}, a_{22}\}$.
\item[{\rm (iv)}]
An eigenvector $v$ corresponding to $\lambda_2$ has its coordinates \textup{(}nonzero and\textup{)} of opposite signs.
\end{enumerate}
\end{proposition}
\begin{proof}
(i)
The characteristic polynomial of $A$ has the form
\begin{equation*}
p_{A}(\lambda) = \lambda^2 - (a_{11} + a_{22}) \lambda + (a_{11} a_{22} - a_{12} a_{21}),
\end{equation*}
with discriminant
\begin{equation*}
\Delta = (a_{11} + a_{22})^2 - 4 (a_{11} a_{22} - a_{12} a_{21}) = (a_{11} - a_{22})^2 + 4 a_{12} a_{21} > 0.
\end{equation*}
Consequently $A$ has two real eigenvalues, $\lambda_2 < \lambda_1$.

\smallskip
(ii)
Since $\lambda_1 + \lambda_2 = a_{11} + a_{22}$, we have $\lambda_1 > \frac{1}{2}(a_{11}+a_{22})$.

Now, if $a_{11} \ge a_{22}$, notice that $[1 \ \ \frac{a_{21}}{\lambda_1 - a_{22}}]^{\top}$, where $\frac{a_{21}}{\lambda_1 - a_{22}} > 0$, is an eigenvector of $A$ corresponding to $\lambda_1$.  If $a_{11} < a_{22}$ then $[\frac{a_{12}}{\lambda_1 - a_{11}} \ \ 1 ]^{\top}$, where $\frac{a_{12}}{\lambda_1 - a_{11}} > 0$, is an eigenvector of $A$ corresponding to $\lambda_1$.

\smallskip
(iii)
For $u = [u_1 \ \ u_2]^{\top}$ we have $(a_{11} - \lambda_1) u_{1} + a_{12} u_{2} = a_{21} u_{1} + (a_{22} - \lambda_1) u_{2} = 0$, so, since $u_{1}$ and $u_{2}$ have the same sign, we must have $\lambda_1 > a_{11}$ and $\lambda_1 > a_{22}$.

The other inequality follows from the first one by the fact that $\lambda_1 + \lambda_2 = a_{11} + a_{22}$.

\smallskip
(iv)
For $v = [v_1 \ \ v_2]^{\top}$, observe that $(a_{11} - \lambda_2) v_{1} + a_{12} v_{2} = 0$, and apply the second inequality in (iii).
\end{proof}

The larger eigenvalue, $\lambda_1$, of $A \in \mathcal{M}$ will be called the {\em principal eigenvalue\/} of $A$ (sometimes the terms {\em dominant\/}, {\em leading\/}, or {\em Perron eigenvalue\/} are used).  An eigenvector $u$ of $A$ pertaining to the principal eigenvalue will be called a {\em principal eigenvector\/} of $A$. When speaking of a principal eigenvector we always assume that both its coordinates are positive.

A principal eigenvector of length one is called {\em normalized}.  A normalized principal eigenvector of a matrix in $\mathcal{M}$ is unique.

\bigskip
The following will be needed in Subsubsection~\ref{subsubsec:analysis}.
\begin{lemma}
\label{lm:monotone-dependence}
The principal eigenvalue of a matrix in $\mathcal{M}$ is a strongly increasing function of any of its entries.
\end{lemma}
\begin{proof}
Recall that the principal eigenvalue of $A = [a_{ij}]_{i, j = 1}^n \in \mathcal{M}$ is given by
\begin{equation*}
\lambda_1 = \frac{a_{11} + a_{22} + \sqrt{(a_{11} - a_{22})^2 + 4 a_{12} a_{21}}}{2}
\end{equation*}
The strongly monotone dependence of $\lambda_1$ on $a_{12}$ or on $a_{21}$ is straightforward.  To prove the dependence on $a_{11}$, observe that
\begin{equation*}
\frac{\partial }{\partial (a_{11})} \lambda_1 = \frac{1}{2} \biggl( 1 + \frac{a_{11} - a_{22}}{\sqrt{(a_{11} - a_{22})^2 + 4 a_{12} a_{21}}} \biggr).
\end{equation*}
But, as $a_{12} a_{21} > 0$, one has
\begin{equation*}
\frac{a_{11} - a_{22}}{\sqrt{(a_{11} - a_{22})^2 + 4 a_{12} a_{21}}} > - 1,
\end{equation*}
which gives that the above partial derivative is bigger than zero.
\end{proof}

\subsection{Strongly cooperative systems of linear ODEs}
\label{subsec:strongly-cooperative}

Recall that a system of linear ODEs
\begin{equation}
\label{eq:cooperative}
x' = A(t)x
\end{equation}
is called {\em strongly cooperative\/} if for each $t \in J$ the matrix $A(t)$ belongs to $\mathcal{M}$.  In this subsection we assume that the matrix function $A(\cdot)$ is continuous.

Now we will give the two-dimensional version of the M\"uller--Kamke theorem.  It is formulated in the linear setting, however a closer inspection shows that its proof is rather nonlinear in the spirit.
\begin{theorem}
\label{thm:strongly-monotone}
Assume that system~\textup{\ref{eq:cooperative}} is strongly cooperative.  Then $X(t; s) \in \mathcal{P}$ for any $s < t$, $s, t \in J$.
\end{theorem}
\begin{proof}
Fix an initial moment $s \in J$.  We start by noting that the first column of the matrix $X(t; s)$ is the value at time $t$ of the solution $[x_1(t) \ \ x_2(t)]^{\top}$ of system $x' = A(t) x$ satisfying the initial condition $x(s) = [1 \ \ 0]^{\top}$.  It follows from the uniqueness of the initial value problem for linear systems of ordinary differential equations that for any $t \in J$ both $x_1(t)$ and $x_2(t)$ cannot be simultaneously equal to zero.

Since $[x_1(t) \ \ x_2(t)]^{\top}$ satisfies the system~\ref{eq:cooperative}, we have $x'_2(s) = a_{21}(s) x_1(s) + a_{22}(s) x_2(s) > 0$, consequently $x_2(t) > 0$ for $t$ sufficiently close to $s$, $t > s$, $t \in J$.  By continuity, since $x_1(s) > 0$, $x_1(t) > 0$ for $t$ sufficiently close to $s$, $t \in J$.  At any rate, there exists $\tau > s$ such that $x_1(t) > 0$ and $x_2(t) > 0$ for all $t \in (s, \tau)$.  We claim that $\tau = \sup{J}$, that is, $x_1(t) > 0$ and $x_2(t) > 0$ for all $t \in J$, $t > s$.  Indeed, suppose to the contrary that this is not so, that is, there exists $\vartheta > s$ such that $x_1(\vartheta) \le 0$ or $x_2(\vartheta) \le 0$. As the product of the functions $x_1(t)$ and $x_2(t)$ is continuous, it follows from the Intermediate Value Theorem that the set $\{\, t > s: x_1(t) x_2(t) = 0 \,\}$ is nonempty.  Specialize $\tau$ to be the greatest lower bound of this set, and assume, for definiteness, that $x_1(\tau) = 0$ (therefore $x_2(\tau) > 0$). $\tau$ cannot be equal to $s$, because we have already shown that $x_1(t) > 0$ directly to the right of $s$. Consequently $\tau > s$, so $x_1(t) > 0$ for $t < \tau$, $t$ sufficiently close to $\tau$, from which it follows that $x'_1(\tau) \le 0$.  But $x'_1(\tau) = a_{11}(\tau) x_1(\tau) + a_{12}(s) x_2(\tau) > 0$, a contradiction.

We have thus shown that the first column of the matrix $X(t; s)$, has, for all $t > s$, positive entries.  By applying a similar reasoning to the solution of system $x' = A(t) x$ satisfying the initial condition $x(s) = [0 \ \ 1]^{\top}$ we show that the second column of the matrix $X(t; s)$, has, for all $t > s$, positive entries, too.
\end{proof}

The full strength of \ref{thm:strongly-monotone} will be needed only in Section~\ref{sec:continuous}.  In the main part, Section~\ref{sec:construction}, we have matrices independent of time. In such a case we can give an alternative proof, using the theory of matrix exponentials only.
\begin{theorem}
\label{thm:strongly-monotone-nonautonomous}
Let $A \in \mathcal{M}$.  Then $e^{tA} \in \mathcal{P}$ for all $t > 0$.
\end{theorem}
\begin{proof}
Assume first that $A \in \mathcal{P}$.  Then $t^{k} A^{k} \in \mathcal{P}$ for all $t > 0$ and $k \in \NN$, consequently $e^{tA} \in \mathcal{P}$.

If $A$ belongs only to $\mathcal{M}$ but not to $\mathcal{P}$, we put $\tilde{A} := aI + A$, where $a := 1 - \min\{a_{11}, a_{22}\}$.  Then $\tilde{A} \in \mathcal{P}$ and, by the previous paragraph, $e^{t \tilde{A}} \in \mathcal{P}$.  As $(aI) A = A (aI)$, there holds $e^{t \tilde{A}} = e^{at} e^{tA}$ (see \ref{eq:commute}), from which it follows immediately that $e^{tA} \in \mathcal{P}$.
\end{proof}

\textit{Remark.}
One could be tempted to use the approach applied in the proof of \ref{thm:strongly-monotone-nonautonomous} in proving \ref{thm:strongly-monotone}.  But this is not so: the obstacle is that $X(t; s)$ for system $x' = A(t) x$ need not be equal to $\exp(\int_{s}^{t} A(\tau) \, d\tau)$.

\subsection{The action of $e^{tA}$ on the unit circle, continued}
\label{subsec:action-on-unit-circle_continued}
In the present subsection we continue the analysis initiated in Subsection~\ref{subsec:action-on-unit-circle}.

Again, the present subsection can be skipped, because it will be helpful only in heuristic considerations why we have chosen such an example.

We assume that $A \in \mathcal{M}$.  As $e^{(t-s)A}$ is the transition matrix of the system $x' = Ax$, \ref{thm:strongly-monotone-nonautonomous} gives that $e^{tA}
\in \mathcal{P}$ for all $t > 0$.

\medskip
Recall that a nonzero $x = [x_1 \ \ x_2] \in \RR^2_{+}$ can be written as $x = r \, [\cos{(\theta)} \ \ \sin{(\theta)}]^{\top} $, where $r = \norm{x} = \sqrt{x \cdot x}$ and $\theta \in [0, \pi/2]$ is given by
\begin{equation*}
\theta =
\begin{cases}
\tan^{-1}{(\frac{x_2}{x_1})} & \text{if } x_1 > 0
\\
\pi/2 & \text{if } x_1 = 0.
\end{cases}
\end{equation*}

We introduce the following notation: $\mathbb{S}_{+} := \mathbb{S} \cap \RR^{2}_{+}$, and $\mathbb{S}_{++} := \mathbb{S} \cap \RR^{2}_{++}$.  Members of $\mathbb{S}_{+}$ can be uniquely written as $y = [\cos{(\theta)} \ \ \sin{(\theta)}]^{\top}$, where $\theta \in [0, \pi/2]$.

For $y \in \mathbb{S}_{+}$ we denote
\begin{equation*}
G(y) := \bigl( A - (A y \cdot y) I \bigr) y.
\end{equation*}
Recall that $y' = G(y)$ is a system of ODEs satisfied by the directions of the solutions of $x' = Ax$ (see \ref{eq:polar-angle}).  We already know (see Subsubsection~\ref{subsubsect:acting-on-directions}) that $G(y)$ is perpendicular to $y$, and that $G(y)$ equals the zero vector if and only if $y$ is the normalized principal eigenvector $u$ of $A$. Otherwise, $G(y)$ is a nonzero vector, pointing either clockwise or counterclockwise.

We check that
\begin{equation*}
G([1 \ \ 0]^{\top}) = [0 \ \ a_{21}]^{\top},
\end{equation*}
so it points counterclockwise, and that
\begin{equation*}
G([0 \ \ 1]^{\top}) = [a_{12} \ \ 0]^{\top},
\end{equation*}
so it points clockwise.

Notice that for any $[y_1 \ \ y_2]^{\top}  \in \mathbb{S}_{+} $, the vector $[y_2 \ \ -\!y_1]^{\top} $ is perpendicular to $[y_1 \ \ y_2]^{\top} $ and points clockwise.  We have thus a simple criterion:
\begin{itemize}
\item
$G(y)$ points clockwise if and only if $G(y) \cdot [y_2 \ \ -\!y_1]^{\top} > 0$,
\item
$G(y)$ points counterclockwise if and only if $G(y) \cdot [y_2 \ \ -\!y_1]^{\top} < 0$.
\end{itemize}

We want to show that for any $y \in \mathbb{S}_{+}$ situated between $[1 \ 0]^{\top}$ and $u$, the vector $G(y)$ points counterclockwise toward $u$, and for any $y \in \mathbb{S}_{+}$ situated between $[0 \ \ 1]^{\top}$ and $u$, the vector $G(y)$ points clockwise toward $u$.

In order to prove that notice first that $[y_2 \ \ -\!y_1]^{\top} = D [y_1 \ \ y_2]^{\top}$, where
\begin{equation*}
D = \begin{bmatrix} 0 & 1 \\
-1 & 0
\end{bmatrix}.
\end{equation*}
Now, the composite function
\begin{equation*}
[0, \pi/2] \ni \theta \mapsto y = [\cos{(\theta)} \ \ \sin{(\theta)}]^{\top} \mapsto G(y) \cdot Dy \in \RR
\end{equation*}
is continuous, takes the value zero only at one $\theta_0 \in (0, \pi/2)$ such that $u = [\cos{(\theta_0)} \ \ \sin{(\theta_0)}]^{\top}$, is positive for $\theta = 0$ and negative for $\theta = \pi/2$. Consequently, it must take positive values for $\theta \in [0, \theta_0)$ and negative values for $\theta \in (\theta_0, \pi/2]$.

\medskip
Consequently, if $A \in \mathcal{M}$ then for any nontrivial solution $x(t)$ of $x' = A x$ such that $x(s) \in \RR^{2}_{+}$ we have the following alternative.
\begin{itemize}
\item
The directions $x(t)/\norm{x(t)}$ are constantly equal to $u$; then $x(t) = {\alpha} e^{{\lambda}_1 t} u$ for some $\alpha > 0$. This occurs when $x(s)/\norm{x(s)} = u$.
\item
For $t > s$ the directions $x(t)/\norm{x(t)}$ tend clockwise to $u$ as $t \to \infty$.  This occurs when $x(s)/\norm{x(s)}$ lies between $[0 \ 1]^{\top}$ and $u$.
\item
For $t > s$ the directions $x(t)/\norm{x(t)}$ tend counterclockwise to $u$ as $t \to \infty$.  This occurs when $x(s)/\norm{x(s)}$ lies between $[1 \ 0]^{\top}$ and $u$.
\end{itemize}

The bottom line is that at each time the directions of the solution tend toward the principal eigenvector.

\section{Construction}
\label{sec:construction}

In the present section we give a construction of a nonautonomous (piecewise constant) planar linear system $x' = A(t)x$ of ODEs such that for each $t \in \RR$ the larger eigenvalue of $A(t)$ equals $-1/2$ but there is a solution not converging to zero as $t \to \infty$.

\subsection{Idea of the construction}
\label{subsec:construction-idea}

We consider a system of linear ODEs
\begin{equation}
\label{eq:main-ex}
x' = A(t) x,
\end{equation}
with $A(t)$ defined as
\begin{equation*}
A(t) :=
\begin{cases}
A^{(1)} & \quad t \in [2k,2k+1), \\
A^{(2)} & \quad t \in [2k+1,2k+2),
\end{cases}
\quad k \in \ZZ,
\end{equation*}
where $A^{(1)}, A^{(2)}$ are $2$ by $2$ matrices.

Notice that $A(t)$ has discontinuity points at integers.

A solution of the system~\ref{eq:main-ex} is defined in the following way:  It is a continuous function $\xi \colon \RR \to \RR^2$ such that
\begin{itemize}
\item
$\xi'(t) = A^{(1)} \xi(t)$, $t \in (2k, 2k+1)$, $k \in \ZZ$;
\item
$\xi'(t) = A^{(2)} \xi(t)$, $t \in (2k+1, 2k+2)$, $k \in \ZZ$;
\item
$\xi'_{-}(2k) = A^{(2)} \xi(2k)$, $\xi'_{+}(2k) = A^{(1)} \xi(2k)$, for any $k \in \ZZ$;
\item
$\xi'_{-}(2k+1) = A^{(1)} \xi(2k+1)$, $\xi'_{+}(2k) = A^{(1)} \xi(2k+1)$, for any $k \in \ZZ$.
\end{itemize}
It is straightforward to see that for any $s \in \RR$ and any $x_0 \in \RR^2$ there exists a unique solution $x(t; s, x_0)$ of \ref{eq:main-ex} satisfying the initial condition $x(s) = x_0$.  Further, we can define the transition matrix $X(t;s)$ as
\begin{equation*}
X(t;s) x_0 = x(t; s, x_0).
\end{equation*}
The transition matrix has all the properties mentioned earlier, in Subsection~\ref{subsec:ODE-systems}:
\begin{enumerate}
\item
$X(s; s) = I$, for any $s \in J$;
\item
$X(r;s) = X(r; t) X(t; s)$, for any $s, t, r \in \RR$;
\item
$X^{-1}(t;s) = X(s; t)$, for any $s, t \in \RR$.
\item
$\displaystyle \frac{\partial}{\partial t} X(t;s) = A(t) X(t;s)$,
for any $s, t \in J$,
\end{enumerate}
except that at $t$ or $s$ being integers its one-sided derivatives satisfy the suitable equalities.

Observe that on a time interval $J$ not containing an integer in its interior a solution of \ref{eq:main-ex} satisfies either the system
\begin{equation*}
x' = A^{(1)} x
\end{equation*}
(when $J \subset [2k, 2k+1]$), or the system
\begin{equation*}
x' = A^{(2)} x
\end{equation*}
(when $J \subset [2k+1, 2k+2]$).  Now, an application of \ref{eq:transition-autonomous} gives that
\begin{equation*}
X(t;s) =
\begin{cases}
\exp\bigl((t-s) A^{(1)}\bigr) & \quad \text{for } s, t \in [2k, 2k+1]
\\
\exp\bigl((t-s) A^{(2)}\bigr) & \quad \text{for } s, t \in [2k+1, 2k+2]
\end{cases}
\end{equation*}

When we restrict ourselves to the interval $[0, 2]$, we have
\begin{equation*}
X(t;0) = \begin{cases}
\exp\bigl(t A^{(1)}\bigr) & \quad \text{for } t \in [0, 1]
\\
\exp\bigl((t - 1) A^{(2)}\bigr) \exp{A^{(1)}} & \quad \text{for } t \in (1, 2]
\end{cases}.
\end{equation*}

As the matrix function $A(t)$ is periodic with period $2$ we have
\begin{equation}
\label{eq:periodic}
X(2k + 2; 2k) = X(2; 0)
\end{equation}
for any $k \in \ZZ$.

We denote
\begin{equation*}
P = X(2; 0) = e^{A^{(2)}} e^{A^{(1)}}.
\end{equation*}
(The letter $P$ stands for Poincar\'e: indeed, $P$ is the Poincar\'e (period) map of the time-periodic system~\ref{eq:main-ex}.)  As a consequence of~\ref{eq:periodic} we obtain
\begin{equation}
\label{eq:periodic-1}
X(2k; 0) = X(2; 0)^k = P^k
\end{equation}
for any $k = 1, 2, 3, \dots$.

\begin{remark}
{\em It can be proved that
\begin{equation*}
X(t; s) = X(t + 2k; s + 2k), \qquad s, t \in \RR, \ k \in \ZZ.
\end{equation*}
We will not need, however, the above equality in its full
generality.}
\end{remark}

From now on, we assume that $A^{(1)}$ and $A^{(2)}$ belong to $\mathcal{M}$ (recall that $\mathcal{M}$ stands for the set of $2 \times 2$ matrices with positive off-diagonal entries).

Our program is to find two matrices, $A^{(1)}, A^{(2)} \in \mathcal{M}$, such that their principal eigenvalues are negative, yet the set of those $y \in \mathbb{S}_{++}$ for which $A^{(1)} y \cdot y > 0$ and $A^{(2)} y \cdot y > 0$ is large.  Indeed, then it is quite likely that for some solution $x(t)$ the directions $x(t)/\norm{x(t)}$ will be in that set for quite a large fraction of time (or, which would be the best, always), so the magnitude of that solution grows from time $t = 0$ to time $t = 2$ (and, by periodicity, it must grow to infinity as time goes to infinity).

Where to look for such matrices?  Certainly not among symmetric (Hermitian) matrices, since for such matrices one can prove quite easily that, if the principal eigenvalue of $A$ is negative then $Ay \cdot y < 0$ for all nonzero $y$.  So, a matrix should be far from symmetric.

\subsection{Definition of $A^{(1)}$ and $A^{(2)}$}
\label{subsec:definitions}

We define
\begin{equation*}
A^{(1)} :=
\left[ \begin{array}{cc}
-1 & c
\\[0.5ex]
\dfrac{1}{4c} & -1
\end{array} \right]
\quad \text{and} \quad
A^{(2)} :=
\left[ \begin{array}{cc}
-1 & \dfrac{1}{4c}
\\[1.5ex]
c & -1
\end{array} \right]~.
\end{equation*}
parameterized by a parameter $c > 0$ ($c$ will be taken to be large).  Observe that the larger $c$ is the farther from symmetric the matrices $A^{(1)}$ and $A^{(2)}$ are.

It is easy to see that the eigenvalues of the matrices $A^{(1)}$ and $A^{(2)}$ are $-1/2$ and $-3/2$.

$[1 \ \ \tfrac{1}{2c}]^{\top}$ is an eigenvector of $A^{(1)}$ corresponding to the principal eigenvalue $-1/2$, and $[1 \ \ -\! \tfrac{1}{2c}]^{\top}$ is an eigenvector of $A^{(1)}$ corresponding to the other eigenvalue $-3/2$.

Similarly, $[\tfrac{1}{2c} \ \ 1]^{\top}$ is an eigenvector of $A^{(2)}$ corresponding to the principal eigenvalue $-1/2$, and $[- \tfrac{1}{2c} \ \ 1]^{\top}$ is an eigenvector of $A^{(2)}$ corresponding to the other eigenvalue $-3/2$.

Denote by $u^{(1)}$ the normalized principal eigenvector of $A^{(1)}$,
\begin{equation*}
u^{(1)} = \left[\frac{2c}{\sqrt{1 + 4 c^2}} \ \ \frac{1}{\sqrt{1 + 4 c^2}} \right]^{\top},
\end{equation*}
and $u^{(2)}$ the normalized principal eigenvector of $A^{(2)}$,
\begin{equation*}
u^{(2)} = \left[\frac{1}{\sqrt{1 + 4 c^2}} \ \ \frac{2c}{\sqrt{1 + 4 c^2}} \right]^{\top}.
\end{equation*}

\subsubsection{Why could the example be O.K.?}
\label{subsubsec:heuristics}

We want to show that, under the choice of the matrices $A^{(1)}$ and $A^{(2)}$ as in the previous subsection, it is very likely that there are plenty of solutions $x(t)$ such that their length tends exponentially fast to infinity.

In order to do that, let us look at the set of those $y \in \mathbb{S}_{++}$ such that $A^{(1)} y \cdot y > 0$.  As the matrix $A^{(2)}$ is the transpose of $A^{(1)}$, that set will be equal to the set of those $y \in \mathbb{S}_{++}$ such that $A^{(2)} y \cdot y > 0$.

We have
\begin{multline*}
A^{(1)} y \cdot y = (A^{(1)} [y_1 \ \ y_2]^{\top})^{\top} [y_1 \ \ y_2]^{\top}
\\
= [y_1 \ \ y_2] (A^{(1)})^{\top} [y_1 \ \ y_2]^{\top} = - (y_1)^2 - (y_2)^2 + \bigl( c + \tfrac{1}{4c} \bigr) y_1 y_2.
\end{multline*}
By writing $y \in \mathbb{S}_{++}$ in polar coordinates as $[\cos(\theta) \ \ \sin(\theta)]^{\top}$,  $\theta \in (0, \pi/2)$, we obtain that
\begin{equation*}
A^{(1)} [\cos(\theta) \ \ \sin(\theta)]^{\top} \cdot [\cos(\theta) \ \ \sin(\theta)]^{\top} = -1 + \tfrac{1}{2} \bigl(c + \tfrac{1}{4c} \bigr) \sin{(2\theta)}.
\end{equation*}
After simple calculation we get $A^{(1)} y \cdot y > 0$ if and only if
\begin{equation*}
y = [\cos(\theta) \ \ \sin(\theta)]^{\top}, \quad \text{where }\theta \in \Bigl( \tfrac{1}{2} \sin^{-1}{(\tfrac{8c}{4 c^2 + 1})}, \tfrac{\pi}{2} - \tfrac{1}{2} \sin^{-1}{(\tfrac{8c}{4 c^2 + 1})} \Bigr),
\end{equation*}
provided that $c > 1 + \frac{\sqrt{3}}{2}$.

Now let us apply the knowledge of how the directions of a solution change, as formulated in Subsection~\ref{subsec:action-on-unit-circle_continued}.  Assume that the initial value $x(0)$ is situated somewhere between the principal eigenvectors for $A^{(1)}$ and $A^{(2)}$.  Recall that at each moment the direction tends toward the normalized principal eigenvector at that moment, so from time $t = 0$ to time $t = 1$ the directions tend clockwise toward the normalized principal eigenvector $u^{(1)}$ of $A^{(1)}$.  They can leave the ``red'' set, but again from time $t = 1$ to time $t = 2$ they tend counterclockwise toward the normalized principal eigenvector $u^{(2)}$ of $A^{(2)}$.  By periodicity, the directions oscillate.

\subsubsection{Analysis}
\label{subsubsec:analysis}

The reasoning given in the previous subsubsection cannot be considered a formal proof.  Now we give an analytical solution.

Observe that for the instability it is enough to find \emph{one} solution such that for some sequence of time moments its lengths tend to infinity.

\smallskip
How to look for such a solution?

By \ref{thm:strongly-monotone}, both matrices $e^{A^{(2)}}$ and $e^{A^{(1)}}$ belong to $\mathcal{P}$, their product, that is, $P$, belongs to $\mathcal{P}$, too.  \ref{prop:principal} states that there exists precisely one normalized  principal eigenvector $w$ of $P$ pertaining to the principal eigenvalue, $\mu$, of $P$.

Denote by $w(t)$ the solution of system~\ref{eq:main-ex} taking value $w$ at time $t = 0$.  Since $w$ is an eigenvector of $P$ corresponding to $\mu$ and since, by \ref{eq:periodic-1}, $w(2n) = X(2n;0) w = P^{n} w$ for all $n = 1, 2, \dots$, we have
\begin{equation*}
w(2n) = {\mu}^n w, \quad n = 1, 2, \dots.
\end{equation*}
So it is sufficient to check that the principal eigenvalue $\mu$ of $P$ is larger than one.

\begin{remark}
{\em The Floquet theory \cite{C} states that there is a decomposition
\begin{equation*}
X(t; 0) = Q(t) e^{tR}, \quad t \in \RR,
\end{equation*}
where $Q(t)$ is a time-periodic matrix function (with period $2$) and $R$ is a constant (in~general, complex) matrix. The eigenvalues of $e^{2R}$ are called\/} characteristic multipliers {\em of $x' = A(t) x$, and a $\nu \in \CC$ such that $e^{2 \nu}$ is a characteristic multiplier is called a} Floquet exponent {\em of $x' = A(t) x$.  Generally, Floquet exponents are not defined uniquely.  But in our case $\mu$ is the (positive real) characteristic multiplier, larger than the other one, and its natural logarithm can be called the} principal Floquet exponent {\em of $x' = A(t) x$.}
\end{remark}

We proceed now to computing (or, rather, estimating from below) $\mu$.  We give two alternative proofs:  first, by explicitly computing the matrix $P$, and second, by giving an approximation of $P$ via a partial sum of the Peano--Baker series and showing that ignoring higher\nobreakdash-\hspace{0pt}order terms suffices for the relevant conclusion.

\bigskip
{\em 1. Direct computing.\/}

The exponential of $A^{(1)}$ is given by the formula
\begin{equation*}
\exp(t A^{(1)}) = e^{-t}
\begin{bmatrix}
\cosh(\frac{t}{2}) & 2c \sinh(\frac{t}{2}) \\
\frac{1}{2c} \sinh(\frac{t}{2}) & \cosh(\frac{t}{2})
\end{bmatrix}.
\end{equation*}
One can find this formula by some Computer Algebra System.  However, we prefer to give a more analytical explanation.

We write $A^{(1)} = -I + B$, where
\begin{equation*}
B = \begin{bmatrix}
0 & c \\
\frac{1}{4c} & 0
\end{bmatrix}.
\end{equation*}
Since $(-I) B = B (-I)$, we can write $e^{t A^{(1)}} = e^{t (-I)} e^{tB}$ (see \ref{eq:commute}).  We easily get $e^{t (-I)} = e^{-t}I$.  So, the problem boils down to finding $e^{tB}$.

We observe that $B^2 = \tfrac{1}{4}I$.  Consequently,
\begin{equation*}
B^{3} = \frac{1}{4} B, \quad B^{5} = \frac{1}{16} B,
\end{equation*}
and generally
\begin{equation*}
B^{2k} = \frac{1}{2^{2k}} I, \quad B^{2k+1} = \frac{1}{2^{2k}} B.
\end{equation*}
We can write
\begin{gather*}
e^{tB} = I + \frac{tB}{1!} + \frac{t^2 B^2}{2!} + \frac{t^3 B^3}{3!} + \frac{t^4 B^4}{4!} + \frac{t^5 B^5}{5!} + \dots
\\
= \left( 1 + \frac{1}{2!} \Bigl( \frac{t}{2} \Bigr)^2 + \frac{1}{4!} \Bigl( \frac{t}{2} \Bigr)^4 + \dots \right) I
\\
+ 2 \left(\frac{1}{1!} \frac{t}{2} + \frac{1}{3!} \Bigl( \frac{t}{2} \Bigr)^3 + \frac{1}{5!} \Bigl( \frac{t}{2} \Bigr)^5 + \dots \right) B,
\end{gather*}
which is easily seen, by comparing the Maclaurin series expansions, to be equal to $(\cosh(\tfrac{t}{2})) I + (2 \sinh(\tfrac{t}{2})) B$.

\smallskip
Similarly we have
\begin{equation*}
\exp(t A^{(2)}) = e^{-t}
\begin{bmatrix}
\cosh(\frac{t}{2}) & \frac{1}{2c} \sinh(\frac{t}{2})
\\
2c \sinh(\frac{t}{2}) & \cosh(\frac{t}{2})
\end{bmatrix}.
\end{equation*}

\vskip3mm

One has
\begin{equation*}
\begin{aligned}
P = e^{A^{(2)}} e^{A^{(1)}} = {} & e^{-1}
\begin{bmatrix}
\cosh(\frac{1}{2}) & \frac{1}{2c} \sinh(\frac{1}{2})
\\[2ex]
2c \sinh(\frac{1}{2}) & \cosh(\frac{1}{2})
\end{bmatrix}
e^{-1}
\begin{bmatrix}
\cosh(\frac{1}{2}) & 2c \sinh(\frac{1}{2})
\\
\frac{1}{2c} \sinh(\frac{1}{2}) & \cosh(\frac{1}{2})
\end{bmatrix}
\\
= {} & e^{-2}
\begin{bmatrix}
\coshsq{\frac{1}{2}} + \frac{1}{4c^2} \sinhsq(\frac{1}{2}) & (2c + \frac{1}{2c}) \cosh(\frac{1}{2}) \, \sinh(\frac{1}{2})
\\[1ex]
(2c + \frac{1}{2c}) \cosh(\frac{1}{2}) \, \sinh(\frac{1}{2}) & \coshsq(\frac{1}{2}) + 4c^2 \sinhsq(\frac{1}{2})
\end{bmatrix}.
\end{aligned}
\end{equation*}
The principal eigenvalue of the last matrix is, by \ref{prop:principal}(iii), bigger than $\coshsq(\frac{1}{2}) + 4 c^2 \sinhsq(\frac{1}{2})$. As $\sinhsq(\frac{1}{2}) > 0$, we need only to take $c > 0$ so large that the last expression is bigger than $e^2$.

Numerical calculation gives that $\coshsq(\frac{1}{2}) + 4 c^2 \sinhsq(\frac{1}{2}) > e^2$ when $c > 2.37323$, whereas the principal eigenvalue of $P$ is $> 1$ when $c > 2.13834$.

\bigskip
{\em 2. Peano--Baker series.\/}
Another way of estimating $\mu$ is by means of the Peano--Baker series.  For an easily readable background on the Peano--Baker series, see~\cite{BSch}.

To be more specific, we shall consider the system
\begin{equation}
\label{eq:shifted}
\tilde{x}' = B(t) \tilde{x},
\end{equation}
where, for any $t \in \RR$,  $B(t) = A(t) + I$.  In other words,
\begin{equation*}
B(t) =
\begin{cases}
B^{(1)} & \quad t \in [2k,2k+1),
\\
B^{(2)} & \quad t \in [2k+1,2k+2),
\end{cases}
\quad k \in \ZZ,
\end{equation*}
with
\begin{equation*}
B^{(1)} = \begin{bmatrix}
0 & c \\
\frac{1}{4c} & 0
\end{bmatrix}, \quad
B^{(2)} = \begin{bmatrix}
0 & \frac{1}{4c} \\
c & 0
\end{bmatrix}.
\end{equation*}
Let $\tilde{X}(t;s)$ stand for the transition matrix for~\ref{eq:shifted}:  for $t \in \RR$, $\tilde{X}(t;s)x_0$ denotes the value at time $t$ of the solution of~\ref{eq:shifted} taking the value $x_0$ at $s$.

Since
\begin{equation*}
\tilde{X}(t;s) = e^{t - s} X(t; s),
\end{equation*}
and we are interested in the principal eigenvalue, $\mu$, of $P = X(2; 0)$ being larger than $1$, we will be done if we can show that the principal eigenvalue of $\tilde{X}(2; 0) = e^{2} P$ is larger than $e^2$.

The Peano--Baker series is given by the formula
\begin{equation}
\label{eq:PB-series}
\tilde{X}(t; 0) = \sum\limits_{k = 0}^{\infty} J_k(t; 0), \quad t \ge 0,
\end{equation}
where
\begin{equation*}
J_0(t; 0) = I, \quad J_{k + 1}(t; 0) = \int\limits_{0}^{t} B(\tau) J_{k}(\tau; 0) \, d\tau.
\end{equation*}
(The reader knowing the Picard iteration will observe that the above is just the Picard iteration formula for the matrix ordinary differential equation $\tilde{X}' = B(t) \tilde{X}$ with the initial condition $\tilde{X}(0) = I$.)

Under our assumptions on $B(t)$, the above series converges, at $t = 2$, to $\tilde{X}(2; 0)$ (see~\cite[Thm.~1]{BSch}).

Let us write first several terms of the Peano--Baker series~\ref{eq:PB-series}, at $t = 2$,
\begin{multline*}
\tilde{X}(2; 0) = I + \int\limits_{0}^{2} B(t_1) \, dt_1 + \int\limits_{0}^{2} B(t_1) \biggl( \int\limits_{0}^{t_1} B(t_2) \, dt_2 \biggr) \, dt_1
\\
+ \int\limits_{0}^{2} B(t_1) \biggl( \int\limits_{0}^{t_1} B(t_2) \biggl( \int\limits_{0}^{t_2} B(t_3) \, dt_3 \biggr) \, dt_2 \biggr) \, dt_1 + \dots.
\end{multline*}

We have
\begin{multline*}
J_{0}(2; 0) + J_{1}(2; 0) = I +\int\limits_{0}^{2} B(\tau) \, d\tau = I + \int\limits_{0}^{1} B(\tau) \, d\tau + \int\limits_{1}^{2} B(\tau) \, d\tau
\\
= I + B^{(1)} + B^{(2)} =
\begin{bmatrix}
1 & c + \frac{1}{4c}
\\[0.5ex]
c + \frac{1}{4c} & 1
\end{bmatrix},
\end{multline*}
The largest (that is, the principal) eigenvalue of the above matrix is easily seen to be $1 + c + \frac{1}{4c}$.  And just as easily we can see that for $c > 0$ sufficiently large (for $c > \tfrac{1}{2}(e^2 - 1) + \tfrac{1}{2} \sqrt{e^4 - 2 e^2}$) that largest eigenvalue is bigger than $e^2$. (Numerical calculation gives that its suffices to have $c > 6.34968$.)

But what about the remaining terms in the Peano--Baker series?  Indeed, adding them cannot make our estimates worse:  since the entries of the matrices $B^{(1)}$ and $B^{(2)}$ are nonnegative, the integrals occurring in the definitions of higher order terms are also matrices with nonnegative entries, so, by \ref{lm:monotone-dependence}, the principal eigenvalue of the matrix $\tilde{X}(2; 0)$ is not less than $1 + c + \frac{1}{4c}$.

We remark here in~passing that our choice of $\tilde{X}(2; 0)$ rather than $X(2; 0)$ is due to the fact that the diagonal terms of the matrices $A(t)$ are negative, which would make the reasoning as in the above paragraph hardly possible.

\bigskip
{\em 3. Comparison of both methods.\/}
The direct computation gives an explicit form of the matrix $P$, so its principal eigenvalue can be calculated.

One of the advantages of the Peano--Baker series is that it is very versatile: due to the monotone dependence of the principal eigenvalue on the entries of the matrices, one needs only to find the second term in the series, and this reduces to integration.  This is of importance when one wants to construct other, more complicated, examples.

\bigskip
{\em 3. Magnus expansion.\/}
There is still another method of solving a nonautonomous linear system of ordinary differential equations, namely the Magnus expansion, which rests on representing the transition matrix as the exponential of some series composed of integrals of nested matrix commutators.   For more on the classical Magnus expansion, as well as its extensions, like the Floquet--Magnus expansion, see the review paper~\cite{B-C-O-R}, see also~\cite{I}.

It seems that it is a challenging task to apply the Magnus expansion to obtain results as in the present paper.

\section{Continuous matrix function}
\label{sec:continuous}
In contrast to the previous parts, reading the present section requires of the reader having experienced more exposure to ``harder'' mathematical thinking.

One could think that perhaps a phenomenon described above has something to do with the discontinuity at integer times.  Results contained in the present section show that this is not so.

\medskip
We start by recalling that, if for a matrix $C = [c_{ij}]_{i,j=1}^2 \in \RR^{2 \times 2}$ we denote its {\em Euclidean norm\/} as
\begin{equation*}
\norm{C} = \biggl( \sum\limits_{i, j = 1}^{2} (c_{ij})^2 \biggr)^{1/2},
\end{equation*}
then for any $C, D \in \RR^{2 \times 2}$ there holds
\begin{equation}
\label{eq:matrix-norm}
\norm{C D} \le \norm{C} \, \norm{D}
\end{equation}
(see, e.g., \cite[5.6]{HJ}).

Another fact is the {\em Gronwall inequality} (\cite[17.3]{HSD}):
\begin{lemma}
\label{lm:Gronwall}
Assume that $\alpha, \beta \ge 0$ and $f(t)$ is a continuous nonnegative function defined on $[a, b]$ such that
\begin{equation*}
f(t) \le \alpha + \beta \int\limits_{a}^{t} f(\tau) \, d\tau, \qquad
t \in [a, b].
\end{equation*}
Then
\begin{equation*}
f(t) \le \alpha e^{{\beta}(t-a)}, \qquad t \in [a, b].
\end{equation*}
\end{lemma}

We proceed now to the construction.

For $s \in [0,1]$ put
\begin{equation*}
\tilde{A}(s) :=
\begin{bmatrix}
-1 & (1-s)c +  \frac{s}{4c}
\\
sc + \frac{1-s}{4c}  & -1
\end{bmatrix},
\end{equation*}
where $c > 0$.

\smallskip
Denote by $\tilde{\lambda}(s)$ the principal eigenvalue of $\tilde{A}(s)$, and put $\bar{A}(s) := \tilde{A}(s) - (\tilde{\lambda}(s) + \frac{1}{2}) I$.  It is easily seen that the normalized principal eigenvector, $u_s$, of $\bar{A}(s)$ is an eigenvector of $\bar{A}(s)$ pertaining to the eigenvalue $\tilde{\lambda}(s) - (\tilde{\lambda}(s) + \frac{1}{2}) = - \frac{1}{2}$.  As $\bar{A}(s)$ belongs to $\mathcal{M}$, $- \frac{1}{2}$ must be therefore its principal eigenvalue.

Observe that
\begin{equation*}
\bar{A}(0) = A^{(1)}, \quad \bar{A}(1) = A^{(2)},
\end{equation*}
where $A^{(1)}$ and $A^{(2)}$ are as in Section~\ref{sec:construction}.

For $\epsilon \in (0, 1/4)$ we define a matrix function $A_{\epsilon} \colon [0,2] \to \mathcal{M}$ by the formula
\begin{equation*}
A_{\epsilon}(t) =
\begin{cases}
\bar{A}\bigl(\frac{1}{2} - \frac{t}{2\epsilon}\bigr) & \text{ for } t \in
[0,\epsilon]
\\
A^{(1)} & \text{ for } t \in [\epsilon, 1 - \epsilon]
\\
\bar{A}\bigl(\frac{t-1}{2\epsilon} + \frac{1}{2}\bigr) & \text{ for } t \in
[1 - \epsilon, 1 + \epsilon]
\\
A^{(2)} & \text{ for } t \in [1 + \epsilon, 2 - \epsilon]
\\
\bar{A}\bigl(\frac{2-t}{2\epsilon} + \frac{1}{2}\bigr) & \text{ for } t \in [2 -
\epsilon, 2].
\end{cases}
\end{equation*}
The function $A_{\epsilon}$ is continuous, and the principal eigenvalue of $A_{\epsilon}(t)$ is constantly equal to $-1/2$.  We extend the matrix function $A_{\epsilon}$ to the whole of $\RR$ by periodicity (with period $2$).

\smallskip
Let $M := \sup\{\, \norm{\bar{A}(s)}: s \in [0,2] \,\}$.

\medskip
Denote by $X_{\epsilon}(t;s)$ the transition matrix for the system $x' = A_{\epsilon}(t) x$:  $X_{\epsilon}(t;s) x_0$ is the solution of the initial value problem
\begin{equation*}
\begin{cases}
x' = A_{\epsilon}(t) x
\\
x(s) = x_0.
\end{cases}
\end{equation*}

We want to show that the matrices $X_{\epsilon}(2;0)$ converge (entrywise), as $\epsilon \to 0^{+}$, to the matrix $X(2; 0)$ as in Section~\ref{sec:construction}.  In~fact, this is a special case of the continuous dependence of solutions of the initial value problem on parameters, as presented, for~example, in~\cite[Chapter 17]{HSD}, but we prefer to give its (simple) proof here.

By integrating the relevant equations for the transition matrix we see that
\begin{equation}
\label{eq:integral}
X_{\epsilon}(t;s) = I + \int\limits_{s}^{t} A_{\epsilon}(\tau) X_{\epsilon}(\tau; s) \, d\tau, \quad X(t;s) = I + \int\limits_{s}^{t} A(\tau) X(\tau; s) \, d\tau
\end{equation}
for any $s, t \in \RR$.

Consequently, with the help of standard estimates of integrals, together with \ref{eq:matrix-norm}, we obtain
\begin{equation*}
\begin{aligned}
\norm{X_{\epsilon}(t;s)} & \le 1 + \int\limits_{s}^{t} \norm{A_{\epsilon}(\tau)} \, \norm{X_{\epsilon}(\tau; s)} \, d\tau \le 1 + M \int\limits_{s}^{t} \norm{X_{\epsilon}(\tau; s)} \, d\tau, \quad s \le t
\\
\norm{X(t;s)} & \le 1 + \int\limits_{s}^{t} \norm{A(\tau)} \, \norm{X(\tau; s)} \, d\tau \le 1 + M \int\limits_{s}^{t} \norm{X(\tau; s)} \, d\tau, \quad s \le t.
\end{aligned}
\end{equation*}
An application of the Gronwall inequality gives that
\begin{equation}
\label{eq1}
\norm{X_{\epsilon}(t;0)} \le e^{Mt} \text{ and } \norm{X(t;0)} \le e^{Mt} \quad \text{for } t \in [0,2].
\end{equation}

Rearranging \ref{eq:integral} we obtain
\begin{equation*}
X_{\epsilon}(t; 0) - X(t; 0) = \int\limits_{0}^{t} A_{\epsilon}(\tau) \bigl(X_{\epsilon}(\tau; 0) - X(\tau; 0)\bigr) \, d\tau + \int\limits_{0}^{t} \bigl(A_{\epsilon}(\tau) - A(\tau)\bigr) X(\tau; 0) \, d\tau,
\end{equation*}
consequently
\begin{equation*}
\norm{X_{\epsilon}(t;0) - X(t; 0)} \le \int\limits_{0}^{t} \norm{A_{\epsilon}(\tau)} \, \norm{X_{\epsilon}(\tau; 0) - X(\tau; 0)} \, d\tau + \int\limits_{0}^{t} \norm{A_{\epsilon}(\tau) - A(\tau)} \, \norm{X(\tau; 0)} \, d\tau.
\end{equation*}
Taking into account that $\norm{A_{\epsilon}(\tau)} \le M$ and $\norm{X(\tau; 0)} \le e^{M \tau}$ we obtain that
\begin{equation*}
\norm{X_{\epsilon}(t;0) - X(t; 0)} \le M \int\limits_{0}^{t} \norm{X_{\epsilon}(\tau; 0) - X(\tau; 0)} \, d\tau + e^{2 M} \int\limits_{0}^{t} \norm{A_{\epsilon}(\tau) - A(\tau)} \, d\tau
\end{equation*}
for $t \in [0,2]$.

Applying the Gronwall inequality once more gives that
\begin{equation*}
\norm{X_{\epsilon}(t;0) - X(t; 0)} \le e^{2 M} e^{Mt} \int\limits_{0}^{2} \norm{A_{\epsilon}(\tau) - A(\tau)} \, d\tau, \quad t \in [0, 2],
\end{equation*}
in~particular
\begin{equation*}
\norm{X_{\epsilon}(2;0) - X(2;0)} \le e^{4 M} \int\limits_{0}^{2} \norm{A_{\epsilon}(\tau) - A(\tau)} \, d\tau.
\end{equation*}
By construction,
\begin{equation*}
\int\limits_{0}^{2} \norm{A_{\epsilon}(\tau) - A(\tau)} \, d\tau \le 8 M \epsilon,
\end{equation*}
consequently
\begin{equation*}
\norm{X_{\epsilon}(2;0) - X(2;0)} \le 8 M e^{4 M} \epsilon.
\end{equation*}

Observe that it follows from the above inequality that, as $\epsilon \to 0$, all the entries of $X_{\epsilon}(2;0)$ converge to the corresponding entries of $X(2;0) = P$.  As, by \ref{thm:strongly-monotone}, $X_{\epsilon}(2;0)$ belong to $\mathcal{P}$, \ref{prop:principal} implies that the principal eigenvalues, $\mu_{\epsilon}$, of $X_{\epsilon}(2;0)$ converge to the principal eigenvalue, $\mu$, of $P$, which is $> 1$. Consequently, for $\epsilon > 0$ sufficiently close to zero the principal eigenvalue of $X_{\epsilon}(2;0)$ is $> 1$.  It suffices now to take the solution of $x' = A_{\epsilon}(t) x$ taking the value $w_{\epsilon}$ at $t = 0$, where $w_{\epsilon}$ is the normalized principal eigenvector of $X_{\epsilon}(2;0)$.

\subsection{Smoother time dependence}
\label{subsec:smooth}
Repeating an argument from Section~\ref{sec:continuous} we can further approximate continuous matrix functions $A(t)$ by matrix functions that are smooth, for example $C^1$, or even $C^{\infty}$.

\section{Extensions of results}
\label{sec:extensions}

Our construction, whether in Section~\ref{sec:construction} or in Section~\ref{sec:continuous}, apparently gives only one unstable solution.

In~reality, however, one can prove, without much effort, more:
\begin{itemize}
\item
Not only $w(k) = {\mu}^k$ for all $k \in \NN$, but also
\begin{equation*}
\lim\limits_{t \to \infty} \frac{\ln\norm{w(t)}}{t} = \mu.
\end{equation*}
\item
$w(t)$ is by~far not the only solution possessing the above property.  Indeed, if $x(t)$ denotes a nontrivial solution such that its initial value, $x(0)$, is in $\RR^2_{+}$, then the directions $x(t)/\norm{x(t)}$ converge, as $t \to \infty$, to the directions $w(t)/\norm{w(t)}$, that is,
\begin{equation*}
\lim\limits_{t \to \infty} \Biggl\lVert \frac{x(t)}{\norm{x(t)}} - \frac{w(t)}{\norm{w(t)}} \Biggl\lVert = 0
\end{equation*}
(and the exponential rate of convergence is equal to half the natural logarithm of the second eigenvalue of the transition matrix $X(2;0)$), from which it follows that
\begin{equation*}
\lim\limits_{t \to \infty} \frac{\ln\norm{x(t)}}{t} = \mu
\end{equation*}
holds for such $x(t)$, too.
\end{itemize}

The above could be a material for undergraduate work.

\section{Non-periodic systems}
\label{sec:non-periodic}
The present section is independent of Section~\ref{sec:continuous}.  Its reading requires the knowledge of basic calculus only.

A natural question appears: Can one construct a nonautonomous strongly cooperative system that is not periodic in time, but which exhibits the phenomenon as in Section~\ref{sec:construction}?

Recall that, in the time-periodic situation, the analysis of the Poincar\'e map is a powerful (and, simultaneously, simple) tool to draw conclusions regarding the (in)stability of a strongly cooperative system of ODEs.  Indeed, in Section~\ref{sec:construction} it suffices to check that the principal eigenvalue of a (linear) Poincar\'e map (the principal Floquet exponent) is larger than one.  The outlook changes dramatically when we take a next step in generalization, that is, we consider systems that are not periodic in time: there is a theory of the principal spectrum/principal Lyapunov exponent (for a survey see Part IV of \cite{JM}, as well as the references contained therein), but it is quite involved (even for almost periodic systems), and beyond the scope of the present article.

\smallskip
The above is one of the reasons why we have chosen to give a construction of such a non-periodic system as a perturbation of a system which is already known: The starting point is the time-periodic system
\begin{equation*}
x' = A(t) x,
\end{equation*}
where $A(t)$ is either $A(t)$ in Section~\ref{sec:construction} or $A_{\epsilon}(t)$ as in Section~\ref{sec:continuous}.   In either case the principal eigenvalue of $A(t)$ is, at any time $t$, equal to $-\tfrac{1}{2}$.   Recall that we have found a solution $w(t) = (w_1(t)\ \ w_2(t))^{\top}$ such that its magnitude, $\lVert w(t) \rVert$, at times $t = 2, 4, 6, \dots$, diverges to infinity.

The idea is to perturb the matrix function $A(t)$ in a non-periodic way so that the principal eigenvalues of the perturbed matrices $\hat{A}(t)$ are, for each $t \ge 0$, less than $-\tfrac{1}{4}$, and yet there exists a solution of the system
\begin{equation*}
y' = \hat{A}(t) y
\end{equation*}
which does not converge to zero as $t \to \infty$.

We apply the simplest possible perturbation:  write
\begin{equation*}
\hat{A}(t) = A(t) + a(t) I,
\end{equation*}
where $a(t)$ is a continuous non-periodic function such that $0 < a(t) < \tfrac{1}{4}$ for all $t \ge 0$.  (For~instance, if we are looking for an almost periodic perturbation we can take $a(t) =\tfrac{1}{16} (2 + \sin{(t)} + \sin{(\sqrt{2}t)})$.)

It is a standard exercise in linear algebra that for any $t \ge 0$ the principal eigenvalue of $\hat{A}(t)$ equals $-\tfrac{1}{2} + a(t)$, consequently is less than $-\tfrac{1}{4}$.

Denote by $v(t) = (v_1(t)\ \ v_2(t))^{\top}$ the solution of $y' = \hat{A}(t) y$ taking the same value at $t = 0$ as $w(t)$.  We have
\begin{multline*}
v'_1(0) = (-1 + a(0)) v_1(0) + a_{12}(0) v_2(0) > - v_1(0) + a_{12}(0) v_2(0) \\
=  - w_1(0) + a_{12}(0) w_2(0) = w'_1(0)
\end{multline*}
and
\begin{multline*}
v'_2(0) = a_{21}(t) v_1(0)  + (-1 + a(0)) v_2(0)
\\
> a_{21}(0) v_1(0)  - v_2(0) = a_{21}(0) w_1(0)  - w_2(0) = w'_2(0),
\end{multline*}
As a consequence, $v_1(t) > w_1(t)$ and $v_2(t) > w_2(t)$ for $t > 0$ sufficiently close to $0$, say, for $t \in (0, \delta)$, where $\delta > 0$.  We claim that those inequalities hold indeed for all $t > 0$.    Suppose not.  Let then $\tau$ stand for the greatest lower bound of those $t > 0$ for which the inequalities do not hold.  We have $\tau \ge \delta > 0$.  Assume for the sake of definiteness that $w_1(t) < v_1(t)$ and $w_2(t) < v_2(t)$ for all $t \in (0, \tau)$ but $w_1(\tau) = v_1(\tau)$.  We have thus
\begin{multline*}
v'_1(\tau) = (-1 + a(\tau)) v_1(\tau) + a_{12}(\tau) v_2(\tau)
\\
> - v_1(\tau) + a_{12}(t) v_2(\tau) \ge - w_1(\tau) + a_{12}(t) w_2(\tau) = w'_1(\tau),
\end{multline*}
from which we deduce that $v_1(t) < w_1(t)$ for $t < \tau$, sufficiently close to $\tau$.  But this is in~contradiction to the definition of $\tau$.

In~particular, it follows that $\lVert v(2k) \rVert > \lVert w(2k) \rVert $ for $k = 1, 2, \dots$.   As the latter sequence has, as $k \to \infty$, limit infinity, the former sequence must have limit infinity, too.

\section{Discussion}
\label{sec:overview}

In the present section we put our results in the perspective of known results and discuss the relevance of various assumptions made by us during the construction.

\bigskip
We start by comparing our construction with that put forward by Josi\'c and Rosenbaum in~\cite{JoRo}.  There, the authors start by taking a $2$ by $2$ matrix $B$ having negative (real) eigenvalues with eigendirections close to each other.  In such a case, there is a good supply of $x \in \RR^2$ such that $Bx \cdot x > 0$.  Then they construct a nonautonomous system $x' = A(t) x$ by rotating the system $x' = B x$ around the origin at such angular velocity that some solution is being kept, for sufficient amount of time, in the set where $A(t) x \cdot x > 0$, which guarantees that this solution is unstable.

A quick look at the properties listed in Subsection~\ref{subsec:strongly-cooperative} shows that the above construction is impossible in the case of strongly cooperative systems.  Indeed, it is a direct consequence of the Perron--Frobenius theorem (\ref{prop:principal}) that at each $t$ one eigenvector of $A(t)$ must lie in the first quadrant.  So the mechanism in our construction must be different from that in~\cite{JoRo}:  it is an instantaneous (or near instantaneous) change of the eigenvectors which causes the system to be unstable.

\medskip
One could ask:  Why have we chosen to start with considering (time\nobreakdash-\hspace{0pt})periodic systems?  The reason is, at~least, threefold.  First, periodic systems can be considered the simplest form of nonautonomous systems (and remember that for autonomous systems the (in)stability is determined by the eigenvalues of the matrix of the system).  The second reason is that for periodic systems a strong tool is known, namely the Poincar\'e map.  That allows us to give a relatively simple proof of instability, just by calculating the eigenvalues of some easy to obtain matrix.

And last but not the least, when one has in mind that it is biological applications that are the main incentive, the fact that a lot of parameters of the systems are periodic in time is due to seasonal changes in the availability of food, etc.

\smallskip
Having said that, it should be emphasized that analogous constructions could be made for quasiperiodic, almost periodic and, more generally, any dependence on time.  An example of such construction is given in Section~\ref{sec:non-periodic}.

\medskip
In our example in Section~\ref{sec:construction} the switching between the matrices $A^{(1)}$ and $A^{(2)}$ occurs at constant intervals.  It seems that when one allows the switching times to be random variables, it could be possible to construct analogous examples.  That could be a subject both for an undergraduate work as well as of some research.

\smallskip
It should be stressed that a very quick change of the coefficients lies at the core of the phenomenon described.  Indeed, there are results showing that when the matrices $A(t)$ change slowly enough, the stability of the system is determined by the signs of the real parts of their eigenvalues (see~\cite{Solo}, and for extensions to linear systems on time scales, \cite{DaC} and~\cite{P}).

\medskip

\section*{Acknowledgments} I thank Jacek Cicho\'n for his help with \textit{Mathematica}, Kre\v{s}imir Josi\'c for his remarks and Hal Smith for calling my attention to Ref.~\cite{MaSm}.

I am very indebted to three anonymous referees, whose critical remarks have greatly contributed to improving the paper.

This research was supported by the NCN grant Maestro 2013/08/A/ST1/00275.

\end{document}